\documentclass[11pt]{article}
\usepackage{amsfonts,amscd,amssymb, amsmath}
\newtheorem{definition}{Definition}[section]

\newtheorem{proposition}[definition]{Proposition}

\newtheorem{remark}[definition]{Remark}

\newtheorem{theorem}[definition]{Theorem}
\newtheorem{example}[definition]{Example}

\def\rawo\lonra{\longrightarrow}

\def\ot{\otimes}

\newcommand{\selabel}[1]{\label{se:#1}}

\allowdisplaybreaks[4]

\newenvironment{proof}{{\it Proof.}}{\hfill $ \square $ \vskip 4mm}

\begin{document}
\title{A new way to iterate Brzezi\'{n}ski crossed products}
\author{Leonard D\u au\c s \thanks {Research supported by the National Research Foundation - United Arab Emirates and by the United Arab Emirates University(NRF-UAEU) under Grant No. 31S076/2013.}\\
Department of Mathematical Sciences, UAE University\\
PO-Box 15551, Al Ain, United Arab Emirates\\
e-mail: leonard.daus@uaeu.ac.ae
\and Florin Panaite\thanks {Work supported by a grant of the Romanian National 
Authority for Scientific Research, CNCS-UEFISCDI, 
project number PN-II-ID-PCE-2011-3-0635,  
contract nr. 253/5.10.2011.}\\
Institute of Mathematics of the
Romanian Academy\\
PO-Box 1-764, RO-014700 Bucharest, Romania\\
 e-mail: Florin.Panaite@imar.ro}
\date{}
\maketitle
\begin{abstract}
If $A\ot _{R, \sigma }V$ and $A\ot _{P, \nu }W$ are two Brzezi\'{n}ski crossed products 
and $Q:W\ot V\rightarrow V\ot W$ is a linear map satisfying certain properties, we 
construct a Brzezi\'{n}ski crossed product $A\ot _{S, \theta }(V\ot W)$. This construction 
contains as a particular case the iterated twisted tensor product of algebras. 
\end{abstract}
%%%%%%%%%%%%%%%%%%%%%%%%%%%%%%
\section*{Introduction}
%%%%%%%%%%%%%%%%%%%%%%%%%%%%%%
${\;\;\;}$The {\em twisted tensor product} of the associative unital algebras $A$ and $B$ 
is a new associative unital algebra structure built on the linear space $A\ot B$ with the help 
of a linear map $R:B\ot A\rightarrow A\ot B$ called a {\em twisting map}. This construction, 
denoted by $A\ot _RB$, appeared in several contexts and has various applications 
(\cite{Cap}, \cite{VanDaele}). Concrete examples come especially from Hopf algebra theory, 
like for instance the smash product. 

It was proved in \cite{jlpvo} that twisted tensor products of algebras may be iterated. Namely, if 
$A\otimes _{R_1}B$, $B\otimes _{R_2}C$ and $A\otimes _{R_3}C$ are twisted tensor products and 
the twisting maps $R_1$, $R_2$, $R_3$ satisfy the braid relation $(id_A\otimes R_2)\circ 
(R_3\otimes id_B)\circ (id_C\otimes R_1)=(R_1\otimes id_C)\circ (id_B\otimes R_3)\circ 
(R_2\otimes id_A)$, then one can define certain twisted tensor products 
$A\otimes _{T_2}(B\otimes _{R_2}C)$ 
and $(A\otimes _{R_1}B)\otimes _{T_1}C$ that are equal as algebras 
(and this algebra is called the iterated twisted tensor product).

The Brzezi\'{n}ski crossed product, introduced in \cite{brz}, is a common 
generalization of twisted tensor products of algebras and the Hopf crossed product (containing also 
as a particular case the quasi-Hopf smash product introduced in \cite{bpvo}).  If $A$ is an 
associative unital algebra, $V$ is a linear space endowed with a distinguished element $1_V$ and 
$\sigma :V\ot V\rightarrow A\ot V$ and $R:V\ot A\rightarrow A\ot V$ are linear maps 
satisfying certain conditions, then the Brzezi\'{n}ski crossed product is a certain associative 
unital algebra structure on $A\ot V$, denoted by $A\ot _{R, \sigma }V$. 

In \cite{iterated} was proved that  Brzezi\'{n}ski crossed products may be iterated, in the following sense. 
One can define first a ''mirror version'' of the Brzezi\'{n}ski crossed product, denoted by 
$W\overline{\otimes}_{P, \nu }D$ (where $D$ is an associative unital algebra, $W$ is a linear space and 
$P$, $\nu $ are certain linear maps). Examples are twisted tensor products of algebras and the 
quasi-Hopf smash product introduced in \cite{bpv}. Then it was proved that, if 
$W\overline{\otimes}_{P, \nu }D$ and $D\ot _{R, \sigma }V$ are two Brzezi\'{n}ski crossed products 
and $Q:V\ot W\rightarrow W\ot D\ot V$ is a linear map satisfying some conditions, then one can define 
certain Brzezi\'{n}ski crossed products $(W\overline{\otimes}_{P, \nu }D)\ot 
_{\overline{R}, \overline{\sigma }}V$ and $W\overline{\otimes }_{\overline{P}, \overline{\nu }}
(D\ot _{R, \sigma }V)$ that are equal as algebras. Iterated twisted tensor products of algebras 
appear as a particular case of this construction, as well as the so-called quasi-Hopf two-sided smash 
product $A\#H\#B$ from \cite{bpvo}. 

The aim of this paper is to show that Brzezi\'{n}ski crossed products may be iterated in a 
different way, that will also contain as a particular case the iterated twisted tensor product of algebras. 
Namely, we prove that, if $A\ot _{R, \sigma }V$ and $A\ot _{P, \nu }W$ are two Brzezi\'{n}ski 
crossed products and 
$Q:W\ot V\rightarrow V\ot W$ is a linear map satisfying certain properties, then we can define 
two Brzezi\'{n}ski 
crossed products $A\ot _{S, \theta }(V\ot W)$ and 
$(A\ot _{R, \sigma }V)\ot _{T, \eta }W$ that are equal as algebras. 

Our inspiration for looking at this new way of iterating Brzezi\'{n}ski crossed products came 
from the following result in graded ring theory: If $G$ is a group, $R$ is a $G$-graded ring, 
$A$ and $B$ are two finite left $G$-sets, then there exists a ring isomorphism between the 
smash products $R\#(A\times B)$ and $(R\# A)\# B$. This result was obtained in \cite[Corollary 3.2]{leo}, 
and it is useful in the study of the von Neumann regularity of rings of the type $R\# A$, 
cf. \cite{leo}. The smash product $R\# A$ of the $G-$graded ring $R$ by a (finite) left $G-$set $A$ was 
introduced in the paper \cite{nrv} and it is a particular case of a more general construction. If $H$ is a Hopf algebra, $R$ an $H-$comodule algebra and $C$ an $H-$module coalgebra, then we may consider the category $_{R}^{C}\mathcal{M}(H)$ of Doi-Koppinen Hopf modules (i.e. left $R-$modules and left $C-$comodules which satisfy certain compatibility relations). Then, the smash product \mbox{$R\symbol{35} A$} used in \cite{leo} is a particular smash product and it is the first example in the category $_{R}^{C}\mathcal{M}(H)$ (in the case when $H$ is the groupring $k[G]$, $R$ a $G-$graded ring and $C$ the grouplike coalgebra $k[A]$ on a $G-$set $A$).

%%%%%%%%%%%%%%%%%%%%%%%%%%%%%%%
\section{Preliminaries}\selabel{1}
%%%%%%%%%%%%%%%%%%%%%%%%%%%%%%%
${\;\;\;\;}$
We work over a commutative field $k$. All algebras, linear spaces
etc. will be over $k$; unadorned $\ot $ means $\ot_k$. By ''algebra'' we 
always mean an associative unital algebra. The multiplication 
of an algebra $A$ is denoted by $\mu _A$ or simply $\mu $ when 
there is no danger of confusion, and we usually denote 
$\mu _A(a\ot a')=aa'$ for all $a, a'\in A$. The unit of an algebra $A$ is 
denoted by $1_A$ or simply $1$ when there is no danger of confusion. 

We recall from \cite{Cap}, \cite{VanDaele} that, given two algebras $A$, $B$ 
and a $k$-linear map $R:B\ot A\rightarrow A\ot B$, with Sweedler-type notation 
$R(b\ot a)=a_R\ot b_R$, for $a\in A$, $b\in B$, satisfying the conditions 
$a_R\otimes 1_R=a\otimes 1$, $1_R\otimes b_R=1\otimes b$, 
$(aa')_R\otimes b_R=a_Ra'_r\otimes (b_R)_r$, 
$a_R\otimes (bb')_R=(a_R)_r\otimes b_rb'_R$, 
for all $a, a'\in A$ and $b, b'\in B$ (where $r$ and $R$ are two different indices), 
if we define on $A\ot B$ a new multiplication, by 
$(a\ot b)(a'\ot b')=aa'_R\ot b_Rb'$, then this multiplication is associative 
with unit $1\ot 1$. In this case, the map $R$ is called 
a {\em twisting map} between $A$ and $B$ and the new algebra 
structure on $A\ot B$ is denoted by $A\ot _RB$ and called the 
{\em twisted tensor product} of $A$ and $B$ afforded by the map $R$. 

We recall from \cite{brz} the construction of  
Brzezi\'{n}ski's crossed product:
\begin{proposition} (\cite{brz}) \label{defbrz}
Let $(A, \mu , 1_A)$ be an (associative unital) algebra and $V$ a 
vector space equipped with a distinguished element $1_V\in V$. Then 
the vector space $A\ot V$ is an associative algebra with unit $1_A\ot 1_V$ 
and whose multiplication has the property that $(a\ot 1_V)(b\ot v)=
ab\ot v$, for all $a, b\in A$ and $v\in V$, if and only if there exist 
linear maps $\sigma :V\ot V\rightarrow A\ot V$ and 
$R:V\ot A\rightarrow A\ot V$  satisfying the following conditions:
\begin{eqnarray}
&&R(1_V\ot a)=a\ot 1_V, \;\;\;R(v\ot 1_A)=1_A\ot v, \;\;\;\forall 
\;a\in A, \;v\in V, \label{brz1} \\
&&\sigma (1_V\ot v)=\sigma (v\ot 1_V)=1_A\ot v, \;\;\;\forall 
\;v\in V, \label{brz2} \\
&&R\circ (id_V\ot \mu )=(\mu \ot id_V)\circ (id_A\ot R)\circ (R\ot id_A), 
\label{brz3} \\
&&(\mu \ot id_V)\circ (id_A\ot \sigma )\circ (R\ot id_V)\circ 
(id_V\ot \sigma ) \nonumber \\
&&\;\;\;\;\;\;\;\;\;\;
=(\mu \ot id_V)\circ (id_A\ot \sigma )\circ (\sigma \ot id_V), \label{brz4} \\
&&(\mu \ot id_V)\circ (id_A\ot \sigma )\circ (R\ot id_V)\circ 
(id_V\ot R ) \nonumber \\
&&\;\;\;\;\;\;\;\;\;\;
=(\mu \ot id_V)\circ (id_A\ot R )\circ (\sigma \ot id_A). \label{brz5} 
\end{eqnarray}
If this is the case, the multiplication of $A\ot V$ is given explicitly by
\begin{eqnarray*} 
&&\mu _{A\ot V}=(\mu _2\ot id_V)\circ (id_A\ot id_A\ot \sigma )\circ 
(id_A\ot R\ot id_V),
\end{eqnarray*}
where $\mu _2=\mu \circ (id_A\ot \mu )=\mu \circ (\mu \ot id_A)$. 
We denote by $A\ot _{R, \sigma }V$ this algebra structure and 
call it the {\em crossed product} (or Brzezi\'{n}ski crossed product) 
afforded by the data $(A, V, R, \sigma )$.  
\end{proposition}

If  $A\ot _{R, \sigma }V$ is a crossed product, we introduce the 
following Sweedler-type notation:
\begin{eqnarray*}
&&R:V\ot A\rightarrow A\ot V, \;\;\;R(v\ot a)=a_R\ot v_R, \\
&&\sigma :V\ot V\rightarrow A\ot V, \;\;\;\sigma (v\ot v')=\sigma _1(v, v') 
\ot \sigma _2(v, v'), 
\end{eqnarray*} 
for all $v, v'\in V$ and $a\in A$. With this notation, the multiplication of 
 $A\ot _{R, \sigma }V$ reads
\begin{eqnarray*}
&&(a\ot v)(a'\ot v')=aa'_R\sigma _1(v_R, v')\ot \sigma _2(v_R, v'), \;\;\;
\forall \;a, a'\in A, \;v, v'\in V.
\end{eqnarray*}

A twisted tensor product is a particular case of a crossed product 
(cf. \cite{guccione}), namely, if $A\ot _RB$ is a twisted tensor product of 
algebras then $A\ot _RB=A\ot _{R, \sigma }B$, where 
$\sigma :B\ot B\rightarrow A\ot B$ 
is given by $\sigma (b\ot b')=1_A\ot bb'$, for all $b, b'\in B$. 
\begin{remark}
The conditions (\ref{brz3}),  (\ref{brz4}) and (\ref{brz5}) for $R$, $\sigma $ 
may be written 
in Sweedler-type notation respectively as 
\begin{eqnarray}
&&(aa')_R\ot v_R=a_Ra'_r\ot (v_R)_r, \label{brz6} \\
&&\sigma _1(y, z)_R\sigma _1(x_R, \sigma _2(y, z))\ot 
\sigma _2(x_R, \sigma _2(y, z)) \nonumber \\
&&\;\;\;\;\;\;\;\;\;\;=\sigma _1(x, y)\sigma _1(\sigma _2(x, y), z)\ot 
\sigma _2(\sigma _2(x, y), z), \label{brz7}\\
&&(a_R)_r\sigma _1(v_r, v'_R)\ot \sigma _2(v_r, v'_R)
=\sigma _1(v, v')a_R\ot \sigma _2(v, v')_R, \label{brz8}
\end{eqnarray}
for all $a, a'\in A$ and $x, y, z, v, v'\in V,$ where $r$ is another copy of $R$. 
\end{remark}
%%%%%%%%%%%%%%%%%%%%%%%%%%%%%%
\section{The main result and examples}
%%%%%%%%%%%%%%%%%%%%%%%%%%%%%
\setcounter{equation}{0}
%%%%%%%%%%%%%%%%%%%%%%%%%%%%
\begin{theorem} \label{principala}
Let $A\ot _{R, \sigma }V$ and $A\ot _{P, \nu }W$ be two crossed products and 
$Q:W\ot V\rightarrow V\ot W$ a linear map, with notation $Q(w\ot v)=
v_Q\ot w_Q$, for all $v\in V$ and $w\in W$. Assume that the following 
conditions are satisfied: \\[2mm]
(i) $Q$ is unital, in the sense that 
\begin{eqnarray}
&&Q(1_W\ot v)=v\ot 1_W, \;\;\; Q(w\ot 1_V)=1_V\ot w, \;\;\; \forall \;
v\in V, \; w\in W. \label{unitQ}
\end{eqnarray}
(ii) the braid relation for $R$, $P$, $Q$, i.e. 
\begin{eqnarray}
&&(id_A\ot Q)\circ (P\ot id_V)\circ (id_W\ot R)=
(R\ot id_W)\circ (id_V\ot P)\circ (Q\ot id_A), 
\end{eqnarray}
or, equivalently, 
\begin{eqnarray}
&&(a_R)_P\ot (v_R)_Q\ot (w_P)_Q=(a_P)_R\ot (v_Q)_R\ot (w_Q)_P, 
\;\;\; \forall \;a\in A, \; v\in V, \; w\in W. \label{braidcoord}
\end{eqnarray}
(iii) we have the following hexagonal relation between $\sigma $, $P$, $Q$:
\begin{eqnarray}
&&(id_A\ot Q)\circ (P\ot id_V)\circ (id_W\ot \sigma )=
(\sigma \ot id_W)\circ (id_V\ot Q)\circ (Q\ot id_V), 
\end{eqnarray}
or, equivalently, 
\begin{eqnarray}
&&\sigma _1(v, v')_P\ot \sigma _2(v, v')_Q\ot (w_P)_Q=
\sigma _1(v_Q, v'_q)\ot \sigma _2(v_Q, v'_q)\ot (w_Q)_q, \label{PQsigma}
\end{eqnarray}
for all $v, v'\in V$ and $w\in W$, where $q$ is another copy of $Q$. \\[2mm]
(iv) we have the following hexagonal relation between $\nu $, $R$, $Q$:
\begin{eqnarray}
&&(id_A\ot Q)\circ (\nu \ot id_V)=(R\ot id_W)\circ (id_V\ot \nu )\circ 
(Q\ot id_W)\circ (id_W\ot Q), 
\end{eqnarray}
or, equivalently, 
\begin{eqnarray}
&&\nu _1(w, w')\ot v_Q\ot \nu _2(w, w')_Q=\nu _1(w_q, w'_Q)_R\ot 
((v_Q)_q)_R\ot \nu _2(w_q, w'_Q), \label{RQniu}
\end{eqnarray}
for all $v\in V$ and $w, w'\in W$, where $q$ is another copy of $Q$. 

Define the linear maps 
\begin{eqnarray*}
&&S:(V\ot W)\ot A\rightarrow A\ot (V\ot W), \;\;\;
S:=(R\ot id_W)\circ (id_V\ot P), \\
&&\theta :(V\ot W)\ot (V\ot W)\rightarrow A\ot (V\ot W), \\
&&\theta :=(\mu _A\ot id_V\ot id_W)\circ (id_A\ot R\ot id_W)
\circ (\sigma \ot \nu )\circ (id_V\ot Q\ot id_W), \\
&&T:W\ot (A\ot V)\rightarrow (A\ot V)\ot W, \;\;\;
T:=(id_A\ot Q)\circ (P\ot id_V), \\
&&\eta :W\ot W\rightarrow (A\ot V)\ot W, \\
&&\eta (w\ot w')=(\nu _1(w, w')\ot 1_V)\ot \nu _2(w, w'), \;\;\; \forall \; w, w'\in W.
\end{eqnarray*}
Then we have a crossed product $A\ot _{S, \theta }(V\ot W)$ (with respect 
to $1_{V\ot W}:=1_V\ot 1_W$), we have a crossed product 
$(A\ot _{R, \sigma }V)\ot _{T, \eta }W$ and we have an algebra isomorphism 
$A\ot _{S, \theta }(V\ot W)\simeq (A\ot _{R, \sigma }V)\ot _{T, \eta }W$ 
given by the trivial identification.
\end{theorem}
\begin{proof}
We prove first that $A\ot _{S, \theta }(V\ot W)$ is a crossed product, 
i.e. we need to prove the relations (\ref{brz1})-(\ref{brz5}) with $R$ 
replaced by $S$, $\sigma $ replaced by $\theta $ etc. The relations 
(\ref{brz1}) and (\ref{brz2}) follow immediately by (\ref{unitQ}) 
and the relations (\ref{brz1}) and (\ref{brz2}) for $R$, $\sigma $ and 
$P$, $\nu $. Note that the maps $S$ and $\theta $ are defined explicitely by 
\begin{eqnarray*}
&&S(v\ot w\ot a)=(a_P)_R\ot v_R\ot w_P, \\
&&\theta (v\ot w\ot v'\ot w')=\sigma _1(v, v'_Q)\nu _1(w_Q, w')_R\ot 
\sigma _2(v, v'_Q)_R\ot \nu _2(w_Q, w'), 
\end{eqnarray*}
for all $v, v'\in V$, $w, w'\in W$ and $a\in A$. We will denote by $R=r={\mathcal R}=
\overline{R}$, $Q=q=\overline{Q}$ and $P=p$ some more copies of $R$, $Q$ and 
$P$. \\
\underline{Proof of (\ref{brz3})}:\\[2mm]
${\;\;\;}$$S\circ (id_V\ot id_W\ot \mu _A)(v\ot w\ot a\ot a')$
\begin{eqnarray*}
&=&S(v\ot w\ot aa')\\
&=&((aa')_P)_R\ot v_R\ot w_P\\
&\overset{(\ref{brz6})}{=}&(a_Pa'_p)_R\ot v_R\ot (w_P)_p\\
&\overset{(\ref{brz6})}{=}&(a_P)_R(a'_p)_r\ot (v_R)_r\ot (w_P)_p\\
&=&(\mu _A\ot id_V\ot id_W)((a_P)_R\ot (a'_p)_r\ot (v_R)_r\ot (w_P)_p)\\
&=&(\mu _A\ot id_V\ot id_W)\circ (id_A\ot S)
((a_P)_R\ot v_R\ot w_P\ot a')\\
&=&(\mu _A\ot id_V\ot id_W)\circ (id_A\ot S)\circ (S\ot id_A)(v\ot w\ot a\ot a'), 
\;\;\;q.e.d.
\end{eqnarray*}
\underline{Proof of (\ref{brz4})}:\\[2mm]
$(\mu _A\ot id_V\ot id_W)\circ (id_A\ot \theta )\circ 
(S\ot id_V\ot id_W)\circ (id_V\ot id_W\ot \theta )(v\ot w\ot v'\ot w'\ot v''\ot w'')$
\begin{eqnarray*}
&=&(\mu _A\ot id_V\ot id_W)\circ (id_A\ot \theta )\circ 
(S\ot id_V\ot id_W)(v\ot w\ot \sigma _1(v', v''_Q)\nu _1(w'_Q, w'')_R\\
&&\ot 
\sigma _2(v', v''_Q)_R\ot \nu _2(w'_Q, w''))\\
&=&(\mu _A\ot id_V\ot id_W)\circ (id_A\ot \theta )
(([\sigma _1(v', v''_Q)\nu _1(w'_Q, w'')_R]_P)_r\ot v_r\ot w_P\\
&&\ot 
\sigma _2(v', v''_Q)_R\ot \nu _2(w'_Q, w''))\\
&=&([\sigma _1(v', v''_Q)\nu _1(w'_Q, w'')_R]_P)_r
\sigma _1(v_r, (\sigma _2(v', v''_Q)_R)_q)
\nu _1((w_P)_q, \nu _2(w'_Q, w''))_{\mathcal R}\\
&&\ot 
\sigma _2(v_r, (\sigma _2(v', v''_Q)_R)_q)_{\mathcal R}\ot 
\nu _2((w_P)_q, \nu _2(w'_Q, w''))\\
&\overset{(\ref{brz6})}{=}&(\sigma _1(v', v''_Q)_P(\nu _1(w'_Q, w'')_R)_p)_r
\sigma _1(v_r, (\sigma _2(v', v''_Q)_R)_q)
\nu _1(((w_P)_p)_q, \nu _2(w'_Q, w''))_{\mathcal R}\\
&&\ot 
\sigma _2(v_r, (\sigma _2(v', v''_Q)_R)_q)_{\mathcal R}\ot 
\nu _2(((w_P)_p)_q, \nu _2(w'_Q, w''))\\
&\overset{(\ref{brz6})}{=}&(\sigma _1(v', v''_Q)_P)_{\overline{R}}
((\nu _1(w'_Q, w'')_R)_p)_r
\sigma _1((v_{\overline{R}})_r, (\sigma _2(v', v''_Q)_R)_q)
\nu _1(((w_P)_p)_q, \nu _2(w'_Q, w''))_{\mathcal R}\\
&&\ot 
\sigma _2((v_{\overline{R}})_r, (\sigma _2(v', v''_Q)_R)_q)_{\mathcal R}\ot 
\nu _2(((w_P)_p)_q, \nu _2(w'_Q, w''))\\
&\overset{(\ref{braidcoord})}{=}&(\sigma _1(v', v''_Q)_P)_{\overline{R}}
((\nu _1(w'_Q, w'')_p)_R)_r
\sigma _1((v_{\overline{R}})_r, (\sigma _2(v', v''_Q)_q)_R)
\nu _1(((w_P)_q)_p, \nu _2(w'_Q, w''))_{\mathcal R}\\
&&\ot 
\sigma _2((v_{\overline{R}})_r, (\sigma _2(v', v''_Q)_q)_R)_{\mathcal R}\ot 
\nu _2(((w_P)_q)_p, \nu _2(w'_Q, w''))\\
&\overset{(\ref{brz8})}{=}&(\sigma _1(v', v''_Q)_P)_{\overline{R}}
\sigma _1(v_{\overline{R}}, \sigma _2(v', v''_Q)_q)
(\nu _1(w'_Q, w'')_p)_R
\nu _1(((w_P)_q)_p, \nu _2(w'_Q, w''))_{\mathcal R}\\
&&\ot 
(\sigma _2(v_{\overline{R}}, \sigma _2(v', v''_Q)_q)_R)_{\mathcal R}\ot 
\nu _2(((w_P)_q)_p, \nu _2(w'_Q, w''))\\
&\overset{(\ref{PQsigma})}{=}&
\sigma _1(v'_{\overline{Q}}, (v''_Q)_q)_{\overline{R}}
\sigma _1(v_{\overline{R}}, \sigma _2(v'_{\overline{Q}}, (v''_Q)_q))
(\nu _1(w'_Q, w'')_p)_R
\nu _1(((w_{\overline{Q}})_q)_p, \nu _2(w'_Q, w''))_{\mathcal R}\\
&&\ot 
(\sigma _2(v_{\overline{R}}, \sigma _2(v'_{\overline{Q}}, (v''_Q)_q))_R)
_{\mathcal R}\ot 
\nu _2(((w_{\overline{Q}})_q)_p, \nu _2(w'_Q, w''))\\
&\overset{(\ref{brz7})}{=}&
\sigma _1(v, v'_{\overline{Q}})
\sigma _1(\sigma _2(v, v'_{\overline{Q}}), (v''_Q)_q)
(\nu _1(w'_Q, w'')_p)_R
\nu _1(((w_{\overline{Q}})_q)_p, \nu _2(w'_Q, w''))_{\mathcal R}\\
&&\ot 
(\sigma _2(\sigma _2(v, v'_{\overline{Q}}), (v''_Q)_q)_R)
_{\mathcal R}\ot 
\nu _2(((w_{\overline{Q}})_q)_p, \nu _2(w'_Q, w''))\\
&\overset{(\ref{brz6})}{=}&
\sigma _1(v, v'_{\overline{Q}})
\sigma _1(\sigma _2(v, v'_{\overline{Q}}), (v''_Q)_q)
[\nu _1(w'_Q, w'')_p
\nu _1(((w_{\overline{Q}})_q)_p, \nu _2(w'_Q, w''))]_R\\
&&\ot 
\sigma _2(\sigma _2(v, v'_{\overline{Q}}), (v''_Q)_q)_R
\ot 
\nu _2(((w_{\overline{Q}})_q)_p, \nu _2(w'_Q, w''))\\
&\overset{(\ref{brz7})}{=}&
\sigma _1(v, v'_{\overline{Q}})
\sigma _1(\sigma _2(v, v'_{\overline{Q}}), (v''_Q)_q)
[\nu _1((w_{\overline{Q}})_q, w'_Q)
\nu _1(\nu _2((w_{\overline{Q}})_q, w'_Q), w'')]_R\\
&&\ot 
\sigma _2(\sigma _2(v, v'_{\overline{Q}}), (v''_Q)_q)_R
\ot \nu _2(\nu _2((w_{\overline{Q}})_q, w'_Q), w'')\\
&\overset{(\ref{brz8})}{=}&
\sigma _1(v, v'_{\overline{Q}})
\{[\nu _1((w_{\overline{Q}})_q, w'_Q)
\nu _1(\nu _2((w_{\overline{Q}})_q, w'_Q), w'')]_R\}_r
\sigma _1(\sigma _2(v, v'_{\overline{Q}})_r, ((v''_Q)_q)_R)\\
&&\ot 
\sigma _2(\sigma _2(v, v'_{\overline{Q}})_r, ((v''_Q)_q)_R)
\ot \nu _2(\nu _2((w_{\overline{Q}})_q, w'_Q), w'')\\
&\overset{(\ref{brz6})}{=}&
\sigma _1(v, v'_{\overline{Q}})
[\nu _1((w_{\overline{Q}})_q, w'_Q)_R
\nu _1(\nu _2((w_{\overline{Q}})_q, w'_Q), w'')_{\mathcal R}]_r
\sigma _1(\sigma _2(v, v'_{\overline{Q}})_r, (((v''_Q)_q)_R)_{\mathcal R})\\
&&\ot 
\sigma _2(\sigma _2(v, v'_{\overline{Q}})_r, (((v''_Q)_q)_R)_{\mathcal R})
\ot \nu _2(\nu _2((w_{\overline{Q}})_q, w'_Q), w'')\\
&\overset{(\ref{RQniu})}{=}&
\sigma _1(v, v'_{\overline{Q}})
[\nu _1(w_{\overline{Q}}, w')
\nu _1(\nu _2(w_{\overline{Q}}, w')_Q, w'')_{\mathcal R}]_r
\sigma _1(\sigma _2(v, v'_{\overline{Q}})_r, (v''_Q)_{\mathcal R})\\
&&\ot 
\sigma _2(\sigma _2(v, v'_{\overline{Q}})_r, (v''_Q)_{\mathcal R})
\ot \nu _2(\nu _2(w_{\overline{Q}}, w')_Q, w'')\\
&\overset{(\ref{brz6})}{=}&
\sigma _1(v, v'_{\overline{Q}})
\nu _1(w_{\overline{Q}}, w')_R
(\nu _1(\nu _2(w_{\overline{Q}}, w')_Q, w'')_{\mathcal R})_r
\sigma _1((\sigma _2(v, v'_{\overline{Q}})_R)_r, (v''_Q)_{\mathcal R})\\
&&\ot 
\sigma _2((\sigma _2(v, v'_{\overline{Q}})_R)_r, (v''_Q)_{\mathcal R})
\ot \nu _2(\nu _2(w_{\overline{Q}}, w')_Q, w'')\\
&\overset{(\ref{brz8})}{=}&
\sigma _1(v, v'_{\overline{Q}})
\nu _1(w_{\overline{Q}}, w')_R
\sigma _1(\sigma _2(v, v'_{\overline{Q}})_R, v''_Q)
\nu _1(\nu _2(w_{\overline{Q}}, w')_Q, w'')_r\\
&&\ot 
\sigma _2(\sigma _2(v, v'_{\overline{Q}})_R, v''_Q)_r
\ot \nu _2(\nu _2(w_{\overline{Q}}, w')_Q, w'')\\
&=&(\mu _A\ot id_V\ot id_W)(\sigma _1(v, v'_{\overline{Q}})
\nu _1(w_{\overline{Q}}, w')_R\ot 
\sigma _1(\sigma _2(v, v'_{\overline{Q}})_R, v''_Q)
\nu _1(\nu _2(w_{\overline{Q}}, w')_Q, w'')_r\\
&&\ot 
\sigma _2(\sigma _2(v, v'_{\overline{Q}})_R, v''_Q)_r
\ot \nu _2(\nu _2(w_{\overline{Q}}, w')_Q, w''))\\
&=&(\mu _A\ot id_V\ot id_W)\circ (id_A\ot \theta )
(\sigma _1(v, v'_{\overline{Q}})
\nu _1(w_{\overline{Q}}, w')_R\\
&&\ot \sigma _2(v, v'_{\overline{Q}})_R
\ot \nu _2(w_{\overline{Q}}, w')\ot v''\ot w'')\\
&=&(\mu _A\ot id_V\ot id_W)\circ (id_A\ot \theta )\circ 
(\theta \ot id_V\ot id_W)(v\ot w\ot v'\ot w'\ot v''\ot w''), \;\;\;q.e.d.
\end{eqnarray*}
\underline{Proof of (\ref{brz5})}:\\[2mm]
${\;\;\;}$
$(\mu _A\ot id_V\ot id_W)\circ (id_A\ot \theta )\circ (S\ot id_V\ot id_W)\circ (id_V\ot 
id_W\ot S)(v\ot w\ot v'\ot w'\ot a)$
\begin{eqnarray*}
&=&(\mu _A\ot id_V\ot id_W)\circ (id_A\ot \theta )\circ (S\ot id_V\ot id_W)
(v\ot w\ot (a_P)_R\ot v'_R\ot w'_P)\\
&=&(\mu _A\ot id_V\ot id_W)\circ (id_A\ot \theta )((((a_P)_R)_p)_r\ot v_r
\ot w_p\ot v'_R\ot w'_P)\\
&=&(((a_P)_R)_p)_r\sigma _1(v_r, (v'_R)_Q)\nu _1((w_p)_Q, w'_P)_{\mathcal R}
\ot \sigma _2(v_r, (v'_R)_Q)_{\mathcal R}\ot \nu _2((w_p)_Q, w'_P)\\
&\overset{(\ref{braidcoord})}{=}&(((a_P)_p)_R)_r\sigma _1(v_r, (v'_Q)_R)
\nu _1((w_Q)_p, w'_P)_{\mathcal R}
\ot \sigma _2(v_r, (v'_Q)_R)_{\mathcal R}\ot \nu _2((w_Q)_p, w'_P)\\
&\overset{(\ref{brz8})}{=}&\sigma _1(v, v'_Q)
((a_P)_p)_R
\nu _1((w_Q)_p, w'_P)_{\mathcal R}
\ot (\sigma _2(v, v'_Q)_R)_{\mathcal R}\ot \nu _2((w_Q)_p, w'_P)\\
&\overset{(\ref{brz6})}{=}&\sigma _1(v, v'_Q)
[(a_P)_p
\nu _1((w_Q)_p, w'_P)]_R
\ot \sigma _2(v, v'_Q)_R\ot \nu _2((w_Q)_p, w'_P)\\
&\overset{(\ref{brz8})}{=}&\sigma _1(v, v'_Q)
[\nu _1(w_Q, w')a_P]_R
\ot \sigma _2(v, v'_Q)_R\ot \nu _2(w_Q, w')_P\\
&\overset{(\ref{brz6})}{=}&\sigma _1(v, v'_Q)
\nu _1(w_Q, w')_R(a_P)_r
\ot (\sigma _2(v, v'_Q)_R)_r\ot \nu _2(w_Q, w')_P\\
&=&(\mu _A\ot id_V\ot id_W)\circ (id_A\ot S)(\sigma _1(v, v'_Q)
\nu _1(w_Q, w')_R\ot \sigma _2(v, v'_Q)_R\ot \nu _2(w_Q, w')\ot a)\\
&=&(\mu _A\ot id_V\ot id_W)\circ (id_A\ot S)\circ (\theta \ot id_A)
(v\ot w\ot v'\ot w'\ot a),
\end{eqnarray*}
so $A\ot _{S, \theta }(V\ot W)$ is indeed a crossed product. With a similar 
computation one can prove that $(A\ot _{R, \sigma }V)\ot _{T, \eta }W$ is a 
crossed product, so the only thing left to prove is that the multiplications of 
$A\ot _{S, \theta }(V\ot W)$ and $(A\ot _{R, \sigma }V)\ot _{T, \eta }W$ coincide. 
A straightforward computation shows that the multiplication of 
$A\ot _{S, \theta }(V\ot W)$ is given by the formula 
\begin{eqnarray*}
&&(a\ot v\ot w)(a'\ot v'\ot w')=a(a'_P)_{\mathcal R}\sigma _1(v_{\mathcal R}, 
v'_Q)\nu _1((w_P)_Q, w')_r\ot 
\sigma _2(v_{\mathcal R}, 
v'_Q)_r\ot \nu _2((w_P)_Q, w').
\end{eqnarray*}
We compute now the multiplication of $(A\ot _{R, \sigma }V)\ot _{T, \eta }W$:\\[2mm]
${\;\;\;}$$(a\ot v\ot w)(a'\ot v'\ot w')$
\begin{eqnarray*}
&=&(a\ot v)(a'\ot v')_T\eta _1(w_T, w')\ot \eta _2(w_T, w')\\
&=&(a\ot v)(a'_P\ot v'_Q)\eta _1((w_P)_Q, w')\ot \eta _2((w_P)_Q, w')\\
&=&(a\ot v)(a'_P\ot v'_Q)(\nu _1((w_P)_Q, w')\ot 1_V)\ot \nu _2((w_P)_Q, w')\\
&=&(a\ot v)(a'_P\nu _1((w_P)_Q, w')_R\ot (v'_Q)_R)\ot \nu _2((w_P)_Q, w')\\
&=&a[a'_P\nu _1((w_P)_Q, w')_R]_r\sigma _1(v_r, (v'_Q)_R)\ot 
\sigma _2(v_r, (v'_Q)_R)\ot \nu _2((w_P)_Q, w')\\
&\overset{(\ref{brz6})}{=}&a(a'_P)_{\mathcal R}(\nu _1((w_P)_Q, w')_R)_r
\sigma _1((v_{\mathcal R})_r, (v'_Q)_R)\ot 
\sigma _2((v_{\mathcal R})_r, (v'_Q)_R)\ot \nu _2((w_P)_Q, w')\\
&\overset{(\ref{brz8})}{=}&a(a'_P)_{\mathcal R}\sigma _1(v_{\mathcal R}, 
v'_Q)\nu _1((w_P)_Q, w')_r\ot 
\sigma _2(v_{\mathcal R}, 
v'_Q)_r\ot \nu _2((w_P)_Q, w'),
\end{eqnarray*}
and we can see that the two multiplications coincide. 
\end{proof}
\begin{example}{\em 
We recall from \cite{jlpvo} what was called there an iterated twisted tensor 
product of algebras. Let $A$, $B$, $C$ be associative unital algebras, 
$R_1:B\ot A\rightarrow A\ot B$, $R_2:C\ot B\rightarrow B\ot C$, 
$R_3:C\ot A\rightarrow A\ot C$ twisting maps  satisfying the braid 
equation 
\begin{eqnarray*}
&&(id_A\ot R_2)\circ (R_3\ot id_B)\circ (id_C\ot R_1)=
(R_1\ot id_C)\circ (id_B\ot R_3)\circ (R_2\ot id_A). 
\end{eqnarray*}
Then we have an algebra structure on $A\ot B\ot C$ (called the iterated twisted 
tensor product) with unit $1_A\ot 1_B\ot 1_C$ and multiplication 
\begin{eqnarray*}
&&(a\ot b\ot c)(a'\ot b'\ot c')=a(a'_{R_3})_{R_1}\ot b_{R_1}b'_{R_2}\ot 
(c_{R_3})_{R_2}c'.
\end{eqnarray*}

We define $V=B$, $W=C$, $R=R_1$, $P=R_3$, $Q=R_2$ and the linear maps 
\begin{eqnarray*}
&&\sigma :V\ot V\rightarrow A\ot V, \;\;\; \sigma (b\ot b')=1_A\ot bb', \;\;\;
\forall \;b, b'\in V, \\
&&\nu :W\ot W\rightarrow A\ot W, \;\;\;\nu (c\ot c')=1_A\ot cc', \;\;\;
\forall \; c, c'\in W. 
\end{eqnarray*}
Then, for the crossed products $A\ot _{R, \sigma }V=A\ot _{R_1}B$, 
$A\ot _{P, \nu }W=A\ot _{R_3}C$ and the map $Q$, one can check that the 
hypotheses of Theorem \ref{principala} are satisfied and the crossed products 
$A\ot _{S, \theta }(V\ot W)\equiv (A\ot _{R, \sigma }V)\ot _{T, \eta }W$ 
(notation as in Theorem \ref{principala}) coincide with the iterated twisted 
tensor product.}
\end{example}
\begin{example}
Let $A\ot _{R, \sigma }V$ be a crossed product and $W$ an 
(associative unital) algebra. Define the linear maps 
\begin{eqnarray*}
&&P:W\ot A\rightarrow A\ot W, \;\;\; P(w\ot a)=a\ot w, \;\;\;
\forall \; a\in A, \; w\in W, \\
&&\nu :W\ot W\rightarrow A\ot W, \;\;\; \nu (w\ot w')=1_A\ot ww', \;\;\;
\forall \; w, w'\in W, 
\end{eqnarray*}
so we have the crossed product $A\ot _{P, \nu }W$ which is just the ordinary 
tensor product of algebras $A\ot W$. Define the linear map $Q:W\ot V
\rightarrow V\ot W$, $Q(w\ot v)=v\ot w$, for all $v\in V$, $w\in W$. 
Then one can easily check that the 
hypotheses of Theorem \ref{principala} are satisfied. 
\end{example}
%%%%%%%%%%%%%%%%%%%%%%%%

\end{document}